\documentclass{article}

\usepackage{amsthm}
\usepackage{amssymb}
\usepackage{amsmath}

\newtheorem{theorem}{Theorem}[section]
\newtheorem{lemma}[theorem]{Lemma}
\theoremstyle{remark}

\theoremstyle{definition}

\begin{document}

\title{Optimal Control for a Steady State\\ Dead Oil Isotherm Problem\thanks{This is a preprint 
of a paper whose final and definitive form will appear in \emph{Control and Cybernetics}. 
Paper submitted 24-Sept-2012; revised 21-March-2013; accepted for publication 17-April-2013.}}

\author{Moulay Rchid Sidi Ammi$^{1}$\\
\texttt{sidiammi@ua.pt}
\and Agnieszka B. Malinowska$^{2}$\\
\texttt{a.malinowska@pb.edu.pl}
\and Delfim F. M. Torres$^3$\\
\texttt{delfim@ua.pt}}

\date{$^1$Department of Mathematics,
AMNEA Group,
Faculty of Sciences and Techniques,
Moulay Ismail University,
B.P. 509 Errachidia, Morocco\\
$^2$Faculty of Computer Science,
Bia{\l}ystok University of Technology,
15-351 Bia{\l}ystok, Poland\\
$^3$Center for Research and Development in Mathematics and Applications,
Department of Mathematics, University of Aveiro,
3810-193 Aveiro, Portugal}

\maketitle

%%%%%%%%%%%%%%%%%%%%%%%%%%%%%%%%%%%%%%%%%%%%%%%%%%%%%%%%%%%%%%%%%%%%

\begin{abstract}
We study the optimal control of a steady-state dead oil isotherm problem.
The problem is described by a system of nonlinear partial differential equations
resulting from the traditional modelling of oil engineering within
the framework of mechanics of a continuous medium. Existence and regularity
results of the optimal control are proved, as well as
necessary optimality conditions.
\end{abstract}

%%%%%%%%%%%%%%%%%%%%%%%%%%%%%%%%%%%%%%%%%%%%%%%%%%%%%%%%%%%%%%%%%%%%

\smallskip


\smallskip

%%%%%%%%%%%%%%%%%%%%%

\smallskip


\medskip

%%%%%%%%%%%%%%%%%%%%%%

\section{Introduction}

We are interested in the optimal control of the steady-state dead oil isotherm problem:
\begin{equation}
\label{P}
\begin{cases}
- \Delta \varphi(u) = \mathop{\rm div}\left(g(u) \nabla p\right)
& \text{ in }  \Omega,\\
- \mathop{\rm div}\left(d(u) \nabla p\right) = f
& \text{ in } \Omega, \\
\left. u\right|_{\partial \Omega} = 0,  \\
\left. p\right|_{\partial \Omega} = 0,
\end{cases}
\end{equation}
where $\Omega$ is an open bounded domain in $\mathbb{R}^2$ with a
sufficiently smooth boundary. Equations \eqref{P} serve as a model
for an incompressible biphasic flow in a porous medium, with
applications to the industry of exploitation of hydrocarbons.
The reduced saturation of oil is denoted by $u$,
and $p$ is the global pressure. To understand
the optimal control problem that we consider here, some
words about the recovery of hydrocarbons are in order. For a more
detailed discussion about the physical justification of equations
\eqref{P} the reader is referred to \cite{mad,MR2342658,MR2397904}
and references therein. At the time of the first run of a layer,
the flow of the crude oil towards the
surface is due to the energy stored in the gases under pressure in
the natural hydraulic system. To mitigate the consecutive decline
of production and the decomposition of the site, water injections
are carried out, well before the normal exhaustion of the layer.
The water is injected through wells with high pressure, by pumps
specially drilled to this end. The pumps allow the displacement of
the crude oil towards the wells of production. The wells must be
judiciously distributed, which gives rise to a difficult problem
of optimal control: how to choose the best installation sites of
the production wells? This is precisely the question we deal
within this work. These requirements lead us to the following
objective functional:
\begin{equation}
\label{eq:cf}
J(u,p,f) = \frac{1}{2} \left\|u - U\right\|_{2}^2 +
\frac{1}{2} \left\|p - P\right\|_{2}^2 + \frac{\beta_1}{2}
\left\|f\right\|_{2 q_0}^{2 q_0},
\end{equation}
where $2 > q_0 > 1$ and $\beta_1 > 0$ is a coefficient of penalization.
The first two terms in \eqref{eq:cf} make possible to minimize the difference
between the reduced saturation of oil $u$,
the global pressure $p$ and the given data $U$ and $P$, respectively.
Our main goal is to present a method to carry out the optimal
control of \eqref{P} with respect to all the important parameters arising in the process.
More precisely, we seek necessary conditions for the admissible
parameters $u$, $p$ and $f$ to minimize the functional $J$.

Theoretical analysis of the time-dependent dead oil problem with different types
of boundary and initial conditions has received a significant amount of attention.
See \cite{mad} for existence of weak solutions to systems related to \eqref{P},
uniqueness and related regularity results in different settings
with various assumptions on the data. So far, optimal control
of a parabolic-elliptic dead oil system is studied in \cite{romaniaDG-DS}.
Optimal control of a discrete dead oil model is considered in \cite{sd2}.
Here we are interested to obtain necessary optimality conditions
for the steady-state case. This is, to the best of our knowledge,
an important open question.

Several techniques for deriving optimality conditions are
available in the literature of optimal control systems governed
by partial differential equations
\cite{Lions,Lions71,Boris2006,jrw,tomas2006,ss}.
In this work we apply the Lagrangian approach used with success by
Bodart, Boureau and Touzani for an optimal control
problem of the induction heating \cite{Bodart}, and by Lee and Shilkin
for the thermistor problem \cite{LeeShilkin}.

The motivation for our work is threefold.
Firstly, the vast majority of the existing literature on dead oil systems
deal with the parabolic-elliptic system. Considering that the relaxation
time for the saturation of oil $u$ is very small, the time derivative
with respect to the saturation is dropped. Hence we get the system \eqref{P}.
Such a steady-state dead oil  model represents a reasonably realistic situation
where we neglect the time derivative. Secondly, some technical difficulties
when dealing with system \eqref{P} arise and rely on the fact that there
is no information on the time derivative of the reduced saturation of oil
as well as on the pressure. As a result, one cannot use directly
the standard compactness results to obtain strong convergence
of sequences of solutions in appropriate spaces.
This is in contrast with \cite{romaniaDG-DS},
where a fully parabolic system is considered. Thirdly, the choice
of the cost function \eqref{eq:cf} for this time dependent problem seems
to be quite appropriate from the point of view of practical applications.

The paper is organized as follows. In Section~\ref{sec:Pr} we set up
notation and hypotheses. Additionally, we recall two lemmas needed in the sequel.
Our main results are stated and proved in the next two sections.
Under adequate assumptions $(H1)$ and $(H2)$ on the data of the problem,
existence and regularity of the optimal control are proved in Section~\ref{sec:3}.
In Section~\ref{sec:MR}, making use of the Lagrangian approach and assuming further
the hypothesis $(H3)$, we derive necessary optimality conditions for a triple
$\left(\bar{u},\bar{p},\bar{f}\right)$ to minimize \eqref{eq:cf} among all
functions $\left(u,p,f\right)$ verifying \eqref{P}. We end with Section~\ref{con}
of conclusion.

%%%%%%%%%%%%%%%%%%%%%%%%%%%%%%%%%%%%%%%%%%%%

\section{Preliminaries}
\label{sec:Pr}

The following assumptions are needed throughout the paper.
Let $g$ and $d$ be real valued $C^1$-functions and $\varphi$
be a $C^{3}$ function. It is required that
\begin{description}

\item[(H1)] $0 < c_1 \le d(r)$, $g(r)$, $\varphi(r) \le c_2$;
$c_{3} \leq d'(r)$, $\varphi'(r)$, $\varphi''(r) \leq c_{4}$ for all $r \in \mathbb{R}$,
where $c_{i}$, $i= 1, \ldots, 4$, are positive constants.

\item[(H2)] $U$, $P$ $\in L^2(\Omega)$, where  $U$, $P : \Omega \rightarrow \mathbb{R}$.

\item[(H3)] $\left|\varphi'''(r)\right| \le c$ for all $r \in \mathbb{R}$.

\end{description}

Henceforth we use the standard notation for Sobolev spaces:
we denote $\|\cdot\|_{p}=\|\cdot\|_{L^{p}(\Omega)}$
for each $p \in [1,\infty]$ and
\begin{equation*}
W_p^{1}= W_p^{1}(\Omega) := \left\{ u \in L^p(\Omega), \,
\nabla u \in L^p(\Omega) \right\},
\end{equation*}
endowed with the norm $\left\|u\right\|_{W_p^{1}(\Omega)}
= \left\|u\right\|_{p} + \left\|\nabla u\right\|_{p}$;
\begin{equation*}
W_p^{2}=W_p^{2}(\Omega) := \left\{ u \in W_p^{1}(\Omega), \,
\nabla^2 u \in L^p(\Omega) \right\},
\end{equation*}
with the norm $\left\|u\right\|_{W_p^{2}(\Omega)}
= \left\|u\right\|_{W_p^{1}(\Omega)} + \left\|\nabla^2 u\right\|_{p}$.
Moreover, we set
\begin{gather*}
V :=  W_2^{1}(\Omega);\\
W := \left\{ u \in W_{2q}^{2}(\Omega), \,
\left. u\right|_{\partial \Omega} = 0\right\},\\
\Upsilon := \left\{ f \in L^{2 q}(\Omega)  \right\},\\
H := L^{2 q}(\Omega) \times \stackrel{\circ}{W}_{2 q}^{2 - \frac{1}{q}}(\Omega),
\end{gather*}
where $\stackrel{\circ}{W}_{p}^{l}(\Omega)$ is the interior of ${W}_{p}^{l}(\Omega)$.

In the sequel we use the following two lemmas
in order to get regularity of solutions.

\begin{lemma}[\cite{Rodrigues87,Xu}]
\label{meyers}
Let $\Omega\subset{R}^n$ be a bounded domain with a smooth boundary.
Assume that $g\in (L^2(\Omega))^n$ and $a\in C(\bar\Omega)$
with $\min_{\bar\Omega}a>0$. Let $u$ be the weak solution
to the following problem:
\begin{gather*}
-\nabla\cdot(a\nabla u)=\nabla\cdot g \quad\text{in }\Omega\\
u=0 \quad\text{on }\partial\Omega.
\end{gather*}
Then, for each $p>2$, there exists a positive constant $c^*$, depending only on
$n$, $\Omega$, $a$ and $p$, such that if $g\in (L^2(\Omega))^n$, then
\[
\| \nabla u\|_p\leq c^* \left(\|g \|_p+\| \nabla u\|_2\right).
\]
\end{lemma}

\begin{lemma}[\cite{lsu}]
\label{lemma2.3}
For any function $u \in C^{\alpha}(\Omega)
\cap \stackrel{\circ}{W}_2^1(\Omega) \cap W_2^2(\Omega)$
there exist numbers $N_0$ and $\varrho_0$ such that for any
$\varrho \le \varrho_0$ there is a finite covering of $\Omega$
by sets of the type $\Omega_\varrho(x_i)$, $x_i \in \bar{\Omega}$, such that
the total number of intersections of different
$\Omega_{2 \varrho}(x_i) = \Omega \cap B_{2 \varrho}(x_i)$
does not increase $N_0$. Hence, we have the estimate
\begin{equation*}
\left\|\nabla u\right\|_{4}^4
\le c \left\|u\right\|_{C^{\alpha}(\Omega)}^2
\mathrm{\varrho}^{2 \alpha} \left( \left\|\nabla^2 u\right\|_{2}^2
+ \frac{1}{\varrho^2} \left\|\nabla u\right\|_{2}^2\right).
\end{equation*}
\end{lemma}

%%%%%%%%%%%%%%%%%%%%%

\section{Existence and Regularity of Optimal Solutions}
\label{sec:3}

In this section we prove existence and regularity of the optimal control
under assumptions $(H1)$ and $(H2)$ on the data of the problem.

%%%%%%%%%%%%%%%%%%%%%

\subsection{Existence of Optimal Solution}

The following existence theorem is proved using Young's inequality
together with the theorem of Lebesgue and some compactness arguments
of Lions \cite{Lions}. The existence follows from the fact that $J$
is lower semicontinuous with respect to the weak convergence.
Recall that along the text constants $c$ are generic,
and may change at each occurrence.

\begin{theorem}
\label{theorem3.1}
Under the hypotheses (H1) and (H2) there exists a $q > 1$,
depending on the data of the problem, such that the
problem of minimizing \eqref{eq:cf} subject to \eqref{P} has an
optimal solution $\left(\bar{u},\bar{p},\bar{f}\right)$ satisfying
\begin{gather*}
\bar{u} \in W_{q}^{2}(\Omega) \cap L^{2}(\Omega),\\
\bar{p} \in L^2(\Omega) \cap W_{2 q}^{1}(\Omega),
\quad \bar{f} \in L^{2 q_0}(\Omega).
\end{gather*}
\end{theorem}

\begin{proof}
Let $\left(u^{m}, p^{m}, f^{m}\right)
\in W_{2}^{1}(\Omega) \times V \times L^{2q_{0}}(\Omega)$
be a sequence minimizing $J(u, p, f)$. Then we have that
$\left(f^{m}\right)$ is bounded in $L^{2q_{0}}(\Omega)$.
By the second equation of \eqref{P} governed by the global pressure
and a general result of elliptic PDEs \cite{bensoussan}, under our hypotheses we have that
$\nabla p^{m}$ is bounded in $L^{2q}(\Omega)$.  Writing now the first equation of \eqref{P} as
$$
- \mathop{\rm div} \left(\varphi'(u^{m}) \varphi(u^{m})\right)
= \mathop{\rm div}\left(g(u^{m}) \nabla p^{m}\right)
$$
and using Lemma~\ref{meyers}, we obtain $\nabla u^{m} \in L^{2q}(\Omega)$.
Hypotheses allow us to express again the first equation of \eqref{P} as
$$
-\varphi'(u^{m})\triangle u^{m}
- \varphi''(u^{m}) |\nabla u^{m}|^{2}= \mathop{\rm div}(g(u^{m})\nabla p^{m}).
$$
Hence,
$$
\|u^{m}\|_{W_{q}^{2}(\Omega)} \leq c,
$$
where all the constants $c$ are independent of $m$.
Using the Lebesgue theorem and compactness arguments of Lions \cite{Lions},
we can extract subsequences, still denoted by $(p^{m})$, $(u^{m})$
and $(f^{m})$, such that
$$
u^{m} \rightarrow \overline{u} \mbox{ weakly in }  W_{q}^{2}(\Omega),
$$
$$
p^{m} \rightarrow \overline{p} \mbox{ weakly in } W_{2q}^{1}(\Omega),
$$
$$
f^{m} \rightarrow \overline{f} \mbox{ weakly in }
L^{2q_{0}}(\Omega).
$$
Then, by Rellich's theorem, we have
$$
p^{m} \rightarrow \overline{p} \mbox{ strongly in } L^{2}(\Omega).
$$
Therefore, by using these facts and passing to the limit in problem \eqref{P},
it follows from the weak lower semicontinuity of $J$ with respect to the weak convergence,
that the infimum is achieved at $\left(\overline{u}, \overline{p}, \overline{f}\right)$.
\end{proof}

%%%%%%%%%%%%%%%%%%%%%

\subsection{Regularity of Solutions}

Regularity of solutions given by Theorem~\ref{theorem4.1} is
obtained using Young's and Holder's inequalities, the Gronwall lemma,
the De Giorgi–-Nash–-Ladyzhenskaya–-Uraltseva theorem, an estimate
from \cite{ks}, and some technical lemmas that can be found in \cite{lsu}.

\begin{theorem}
\label{theorem4.1}
Let $\left(\bar{u},\bar{p},\bar{f}\right)$ be
an optimal solution to the problem of minimizing \eqref{eq:cf}
subject to \eqref{P}. Suppose that (H1) and (H2) are satisfied.
Then, there exist $\alpha > 0$ such that the following regularity
conditions hold:
\begin{gather}
\bar{u}, \bar{p} \in C^{\alpha}(\Omega), \label{eq:4.1} \\
\bar{u}, \, \bar{p} \in W_{4}^{1}(\Omega), \label{eq:4.2} \\
\bar{u}, \, \bar{p} \in W_{2}^{2}(\Omega), \label{eq:4.3} \\
\bar{u} \in C^{\frac{1}{4}}(\overline{\Omega}), \label{eq:4.5} \\
\bar{u} \in W_{2 q_0}^{2}(\Omega), \quad \bar{p} \in W_{2 q_0}^{2}(\Omega), \label{eq:4.6}
\end{gather}
where $q_{0}$ appears in the cost function \eqref{eq:cf}.
\end{theorem}

\begin{proof}
Firstly, \eqref{eq:4.1} is an immediate application
of the general results of \cite{lsu,Lions,sol}.
To continue the proof of Theorem~\ref{theorem4.1}, we need to estimate
$\|\nabla u\|_{4}$ in function of $\|\nabla p\|_{4}$.
Taking into account the first equation of \eqref{P},
it is well known that $u \in W_{4}^{1}(\Omega)$ (see \cite{ks}) and
\begin{equation}
\label{eq:4.7}
\|\nabla u\|_{4} \leq c \|\nabla p\|_{4}.
\end{equation}
Using Lemma~\ref{lemma2.3}, we have,
for any $\varrho < \varrho_{0}$, that
\begin{equation*}
\left\|\nabla p\right\|_{4}^4 \le c \|p\|_{C^{\alpha}(\bar{\Omega})}^2
\mathrm{\varrho}^{2 \alpha} \left\{ \left\|\nabla p\right\|_{4}^4
+ \frac{1}{\varrho^2} \left\|\nabla p\right\|_{2}^2\right\}.
\end{equation*}
Therefore, we get \eqref{eq:4.2} for an eligible choice of $\varrho$.
Using \eqref{eq:4.7}, we obtain that $u \in W_{4}^{1}(\Omega)$.
On the other hand, by the first equation of \eqref{P} and the regularity \eqref{eq:4.2},
we have that $u \in W_{2}^{2}(\Omega)$. Moreover, it follows, by the fact that
$u \in W_{2}^{2}(\Omega)$, that $p \in W_{2}^{2}(\Omega)$.
Using again \eqref{eq:4.2} and the fact that $W_{4}^{1}(\Omega)
\hookrightarrow C^{\frac{1}{4}}(\overline{\Omega})$, the
regularity estimate \eqref{eq:4.5} follows.
Finally, the right-hand side of the first equation of \eqref{P}
belongs to $L^{4}(\Omega)\hookrightarrow L^{2q_{0}}(\Omega)$ as
$2q_{0} < 4$. Thus, by \eqref{eq:4.3} we get
$u \in W_{2q_{0}}^{2}(\Omega)$. Since $f \in L^{2q_{0}}(\Omega)$,
the same estimate follows from the second equation
of the system \eqref{P} for $p$.
\end{proof}

%%%%%%%%%%%%%%%%%%%%%

\section{Necessary Optimality Conditions}
\label{sec:MR}

We define the following nonlinear operator corresponding to
\eqref{P}:
\begin{gather*}
F : W \times W \times \Upsilon \longrightarrow H \\
\left(u,p,f\right) \longrightarrow F(u,p,f) = 0,
\end{gather*}
where
\begin{equation*}
F(u,p,f) =
\left(
\begin{array}{cc}
   - \Delta \varphi(u) - \mathop{\rm div}(g(u) \nabla p) & \\
   - \mathop{\rm div}\left(d(u) \nabla p\right) - f &  \\
\end{array}
\right).
\end{equation*}
 Due to the estimate
\begin{equation*}
\left\|v\right\|_{W_{\frac{4 q}{2 - q}}^{1}(\Omega)} \le c \left\|v\right\|_{W_{2 q}^{2}(\Omega)},
\quad \forall v \in W_{2 q}^{2}(\Omega),
\quad 1 < q < 2
\end{equation*}
(see \cite{lsu}), hypothesis (H1) and regularity
results (Theorem~\ref{theorem4.1}), we have
\begin{equation*}
\varphi'(u)\Delta u, \varphi''(u) \left|\nabla u\right|^2 , \quad
g'(u) \nabla u \nabla p , \quad d(u) \nabla u \nabla p \in L^{\frac{2
q}{2 - q}}(\Omega) \subset L^{2 q}(\Omega).
\end{equation*}
Thus, it follows that $F$ is well defined.

%%%%%%%%%%%%%%%%%%%%%%%%%%%%%%%%%%%%%%%%%%%%%%%%%%

\subsection{G\^{a}teaux Differentiability}

\begin{theorem}
\label{thm5.1}
Let assumptions (H1) through (H3) hold.
Then, the operator $F$ is G\^{a}teaux differentiable
and its derivative is given by
\begin{multline*}
\delta F(u,p,f)(e,w,h) = \frac{d}{ds} F\left(u + s e, p
+ s w, f + s h\right)\left.\right|_{s=0}
= \left(\delta F_1, \delta F_2\right) \\
= \left(\begin{array}{cc}
- \mathop{\rm div}\left(\varphi'(u)\nabla e\right)
- \mathop{\rm div}\left(\varphi''(u) e \nabla u\right)
- \mathop{\rm div}\left(g(u) \nabla w\right)
- \mathop{\rm div}\left(g'(u) e \nabla p\right) &  \\
- \mathop{\rm div}\left(d(u) \nabla w\right)
- \mathop{\rm div}\left(d'(u) e \nabla p\right) - h &  \\
\end{array}
\right)
\end{multline*}
for all $(e,w,h) \in W \times W \times \Upsilon$.
Furthermore, for any optimal solution
$\left(\bar{u},\bar{p},\bar{f}\right)$ of the problem of
minimizing \eqref{eq:cf} among all the functions
$\left(u,p,f\right)$ satisfying \eqref{P}, the image of $\delta
F\left(\bar{u},\bar{p},\bar{f}\right)$ is equal to $H$.
\end{theorem}

To prove Theorem~\ref{thm5.1} we make use of the following lemma.

\begin{lemma}
\label{lemma5.2}
The operator $\delta F(u,p,f) : W \times W \times
\Upsilon \longrightarrow H$ is  linear and bounded.
\end{lemma}

\begin{proof}
We have for all $(e,w,h) \in W \times W \times \Upsilon$ that
\begin{multline*}
\delta_{p}F_{2}(u, p, f)(e, w, h)
=  - \mathop{\rm div}\left( d(u)
\nabla w\right)- \mathop{\rm div}\left( d'(u)e \nabla p \right)-h\\
= - d(u)\triangle w -d'(u) \nabla u \cdot \nabla w - d'(u)e
\triangle p -d'(u) \nabla e \cdot \nabla u- d'(u)e \nabla u \cdot \nabla p -h,
\end{multline*}
where $\delta_{p} F$ is the G\^ateaux derivative of $F$ with respect
to $p$. Then, using hypothesis $(H1)$, we obtain that
\begin{multline}
\label{eq:4.12}
\|\delta_{p}F_{2}(u, p, f)(e, w, h) \|_{2q}
\leq \| \nabla w \|_{2q}+ c \| \triangle w \|_{2q}\\
+ c  \|\nabla u \cdot \nabla w\|_{2q}
+c \| e \triangle p \|_{2q}+ c  \|\nabla e \cdot \nabla u
\|_{2q} + c  \| e \nabla u \cdot \nabla p\|_{2q}+ \|h\|_{2q}.
\end{multline}
We proceed to estimate the term $\| e \nabla u \cdot \nabla p\|_{2q}$.
Similar arguments can be applied to the remaining terms of \eqref{eq:4.12}.
We have
\begin{equation*}
\begin{split}
\| e \nabla u \cdot \nabla p\|_{2q} &\leq \|e\|_{\infty} \|
\nabla u \cdot \nabla p\|_{2q}\\
&\leq  \|e\|_{\infty} \|\nabla u\|_{\frac{4q}{2-q}} \|\nabla
p\|_{4}\\
&\leq c \|u\|_{W} \|p\|_{W} \|e\|_{W}.
\end{split}
\end{equation*}
Then,
\begin{equation}
\label{eq:4.13}
\|\delta_{p}F_{2}(u, p, f)(e, w, h) \|_{2q} \leq c \left(
\|u\|_{W}, \|p\|_{W}, \|f\|_{\Upsilon} \right) \left( \|e\|_{W}
+\|w\|_{W}+ \|h\|_{\Upsilon} \right).
\end{equation}
On the other hand,
\begin{equation*}
\begin{split}
\delta_{u}F_{1}&(u, p, f)(e, w, h)\\
&= -\mathop{\rm div}\left( \varphi'(u)
\nabla e\right)- \mathop{\rm div}\left( \varphi''(u)e
\nabla u \right)
- \mathop{\rm div}\left( g(u) \nabla w\right)
- \mathop{\rm div}\left( g'(u)e
\nabla p \right)\\
&= - \varphi'(u)\triangle e
-\varphi''(u) \nabla u \cdot \nabla e
-\varphi''(u)e \triangle u -\varphi''(u) \nabla e \cdot \nabla u
- \varphi'''(u)e |\nabla u|^{2}\\
& \quad -g(u)\triangle w- g'(u) \nabla u \cdot \nabla w
- g'(u) e \triangle p- g'(u)\nabla e \cdot \nabla p
-g''(u)e \nabla u \cdot \nabla p,
\end{split}
\end{equation*}
where $\delta_{u} F$ is the G\^ateaux derivative of $F$ with
respect to $u$. The same arguments as above give that
\begin{equation}
\label{eq:4.14}
\|\delta_{u}F_{1}(u, p, f)(e, w, h) \|_{2q} \leq c \left(
\|u\|_{W}, \|p\|_{W}, \|f\|_{\Upsilon} \right) \left( \|e\|_{W}+
\|w\|_{W}+ \|h\|_{\Upsilon} \right).
\end{equation}
Hence, by \eqref{eq:4.13} and \eqref{eq:4.14},
\begin{equation*}
\|\delta F(u, p, f)(e, w, h) \|_{H\times H\times \Upsilon} \leq c
\left( \|u\|_{W}, \|p\|_{W}, \|f\|_{\Upsilon} \right) \left(
\|e\|_{W}+ \|w\|_{W}+ \|h\|_{\Upsilon} \right).
\end{equation*}
Consequently the operator $\delta_{u}F_{1}(u, p, f)$
is linear and bounded.
\end{proof}

\begin{proof}[Proof of Theorem~\ref{thm5.1}]
In order to show that the image of $\delta F(\overline{u},
\overline{p}, \overline{f})$ is equal to $H$, we need to prove
that there exists a $(e, w, h) \in W \times W \times \Upsilon$
such that
\begin{equation}
\label{eq:4.15}
\begin{gathered}
-\mathop{\rm div}\left( \varphi'(\overline{u}) \nabla e\right)
- \mathop{\rm div}\left(\varphi''(\overline{u})e \nabla \overline{u} \right)
- \mathop{\rm div}\left(g(\overline{u}) \nabla w\right)
- \mathop{\rm div}\left( g'(\overline{u})e
\nabla \overline{p} \right)= \alpha,  \\
- \mathop{\rm div}\left( d(\overline{u})\nabla w\right)
- \mathop{\rm div}\left( d'(\overline{u})e
\nabla \overline{p} \right)-h = \beta, \\
\left. e\right|_{\partial \Omega} = 0,  \\
\left. w\right|_{\partial \Omega} = 0,
\end{gathered}
\end{equation}
for any $(\alpha, \beta) \in H$.
Writing the system \eqref{eq:4.15} for $h=0$ as
\begin{equation}
\label{eq:4.16}
\begin{gathered}
-\varphi'(\overline{u})\triangle e -2 \varphi''(\overline{u})\nabla
\overline{u} \cdot \nabla e - \varphi''(\overline{u})e \triangle
\overline{u}-\varphi'''(\overline{u})e |\nabla \overline{u}|^{2}\\
-g(\overline{u}) \triangle w- g'(\overline{u}) \nabla
\overline{u} \cdot \nabla w- g'(\overline{u}) e \triangle \overline{p}
-g'(\overline{u})\nabla \overline{p} \cdot \nabla e - g''(\overline{u})
e \nabla \overline{u} \cdot \nabla \overline{p} = \alpha,\\
-  d(\overline{u}) \triangle w
-d'(\overline{u}) \nabla \overline{u} \cdot \nabla w - d'(\overline{u})
e \triangle \overline{p} - d'(\overline{u}) \nabla \overline{u} \cdot
\nabla \overline{e} -d'(\overline{u}) e \nabla \overline{u} \cdot
\nabla \overline{p} = \beta,\\
\left. e\right|_{\partial \Omega} = 0, \\
\left. w\right|_{\partial \Omega} = 0,
\end{gathered}
\end{equation}
it follows from the regularity of the optimal solution
(Theorem~\ref{theorem4.1}) that
\begin{gather*}
\varphi''(\overline{u}) \triangle \overline{u},\,
\varphi'''(\overline{u}) |\nabla \overline{u}|^{2},\,
g'(\overline{u})  \triangle \overline{p},\,
g''(\overline{u})
\nabla \overline{u} \cdot \nabla \overline{p},\,
d'(\overline{u}) \triangle \overline{p},\,
d'(\overline{u})  \nabla \overline{u}
\cdot \nabla \overline{p}
\in L^{2q_{0}}(\Omega),\\
\varphi''(\overline{u})\nabla \overline{u},
\, g'(\overline{u})\nabla \overline{u},
\, g'(\overline{u})\nabla \overline{p},
\, d'(\overline{u})\nabla \overline{u}
\in L^{4q_{0}}(\Omega).
\end{gather*}
By general results of elliptic PDEs \cite{lsu,Lions,sol}, there
exists a unique solution of system \eqref{eq:4.16} and hence
there exists a $(e, w, 0)$ verifying \eqref{eq:4.15}. We conclude
that the image of $\delta F$ is equal to $H$.
\end{proof}

%%%%%%%%%%%%%%%%%%%%%%%%%%%%%%%%%%%%%%%%%%%%%%%%%%

\subsection{Optimality Condition}

We consider the cost functional
$J: W \times  W \times \Upsilon \rightarrow \mathbb{R}$
\eqref{eq:cf} and the Lagrangian $\mathcal{L}$ defined by
$$
\mathcal{L}\left(u, p, f, p_{1}, e_{1}\right)=
J\left(u, p, f \right)+ \left\langle F(u, p, f),
\left(\begin{array}{cc}  p_{1}   \\
e_{1}
\end{array}\right) \right\rangle,
$$
where the bracket $\langle \cdot, \cdot \rangle$ denote the
duality between $H$ and $H'$.

\begin{theorem}
\label{thm5.2}
Under hypotheses (H1)--(H3), if $\left(\overline{u}, \overline{p},
\overline{f}\right)$ is an optimal solution to the problem of
minimizing \eqref{eq:cf} subject to \eqref{P}, then there exist
functions $\left(\overline{e_{1}}, \overline{p_{1}}\right)
\in W_{2}^{2}(\Omega) \times W_{2}^{2}(\Omega)$ satisfying the
following conditions:
\begin{equation*}
\begin{gathered}
\mathop{\rm div}\left( \varphi'(\overline{u}) \nabla e_{1}\right)
 -d'(\overline{u}) \nabla \overline{p} \cdot \nabla \overline{p_{1}}
-\varphi''\left(\overline{u}\right) \nabla \overline{u} \cdot \nabla
\overline{e_{1}} - g'(\overline{u})\nabla \overline{p} \cdot \nabla
\overline{e_{1}}= \overline{u}-U,\\
\left. \overline{e_{1}}\right|_{\partial \Omega} = 0,
\end{gathered}
\end{equation*}
\begin{equation*}
\begin{gathered}
\mathop{\rm div}\left( d(\overline{u})
\nabla \overline{p_{1}}\right)
+ \mathop{\rm div}\left( g(\overline{u}) \nabla
\overline{e_{1}}\right)=\overline{p}-P, \\
\left. \overline{p_{1}}\right|_{\partial \Omega} = 0,
\end{gathered}
\end{equation*}
\begin{equation}
\label{eq:4.19}
2q_{0} \beta_{1}|\overline{f}|^{2q_{0}-2}\overline{f}
= \overline{p_{1}}.
\end{equation}
\end{theorem}

\begin{proof}
Let $\left( \overline{u}, \overline{p}, \overline{f} \right)$ be
an optimal solution to the problem of minimizing \eqref{eq:cf}
subject to \eqref{P}. It is well known (\textrm{cf., e.g.,}
\cite{fur}) that there exist Lagrange multipliers $\left(
\overline{p_{1}}, \overline{e_{1}} \right) \in H'$ verifying
$\delta_{(u, p, f)}\mathcal{L}
\left(\overline{u}, \overline{p}, \overline{f}, \overline{p_{1}},
\overline{e_{1}} \right)(e, w, h)= 0$ for all $(e, w, h)\in
W \times W \times \Upsilon$, with $\delta_{(u, p, f)}\mathcal{L}$
the G\^{a}teaux derivative of $\mathcal{L}$ with
respect to $(u, p, f)$. We then obtain
\begin{multline*}
\int_{\Omega} \left( (\overline{u}-U)e +(\overline{p}-P)w+ 2q_{0}
\beta_{1}|\overline{f}|^{2q_{0}-2}\overline{f}h
 \right)\, dx \\
+\int_{\Omega} \left(- \mathop{\rm div}\left( \varphi'(\overline{u}) \nabla e
\right)- \mathop{\rm div}\left( \varphi''(\overline{u})e \nabla
\overline{u}\right)- \mathop{\rm div}\left( g(\overline{u}) \nabla w \right) -
\mathop{\rm div}\left( g'(\overline{u}) e \nabla \overline{p}\right)
\right)\overline{e_{1}} \, dx \\
+\int_{\Omega} \left(- \mathop{\rm div}\left( d(\overline{u}) \nabla w \right)-
\mathop{\rm div}\left( d'(\overline{u}) e\nabla \overline{p}\right)-h
\right)\overline{p_{1}}  \, dx =0
\end{multline*}
for all $(e, w, h)\in W \times W \times \Upsilon$.
This last system is equivalent to
\begin{multline*}
\int_{\Omega} \left( (\overline{u}-U)e
-\mathop{\rm div}\left( d'(\overline{u})
e\nabla \overline{p}\right)\overline{p_{1}}
-\mathop{\rm div}\left(
\varphi'(\overline{u}) \nabla e \right) \overline{e_{1}} \right. \\
\left. - \mathop{\rm div}\left( \varphi''(\overline{u})e \nabla
\overline{u}\right) \overline{e_{1}}
- \mathop{\rm div}\left( g'(\overline{u}) e
\nabla \overline{p}\right)\overline{e_{1}} \right) \, dx \\
+\int_{\Omega} \left( (\overline{p}-P)w
- \mathop{\rm div}\left( d(\overline{u}) \nabla w \right)
\overline{p_{1}} - \mathop{\rm div}\left( g(\overline{u}) \nabla w
\right)\overline{e_{1}} \right) \, dx \\
 +\int_{\Omega} \left(2q_{0}
\beta_{1}|\overline{f}|^{2q_{0}-2}\overline{f}h - \overline{p_{1}}
h \right)\, dx=0
\end{multline*}
for all $(e, w, h)\in W \times W \times \Upsilon$. In other words,
\begin{multline}
\label{eq:4.21}
\int_{\Omega} \bigl( \left(\overline{u}-U\right)
+d'\left(\overline{u}\right)\nabla \overline{p}
\cdot \nabla \overline{p_{1}}\\
- \mathop{\rm div}\left( \varphi'(\overline{u}) \nabla \overline{e_{1}}\right)
+ \varphi''\left(\overline{u}\right) \nabla \overline{u} \cdot \nabla
\overline{e_{1}}+g'(\overline{u})\nabla \overline{p}
\cdot \nabla \overline{e_{1}} \bigr) e \, dx\\
+\int_{\Omega} \left( (\overline{p}-P)
- \mathop{\rm div}\left( d(\overline{u}) \nabla
\overline{p_{1}} \right) - \mathop{\rm div}\left( g(\overline{u}) \nabla
\overline{e_{1}} \right) \right)w \, dx \\
+\int_{\Omega} \left(2q_{0}
\beta_{1}|\overline{f}|^{2q_{0}-2}\overline{f}h
- \overline{p_{1}} h \right)\, dx =0
\end{multline}
for all $(e, w, h)\in W \times W \times \Upsilon$.
Consider now the system
\begin{equation}
\label{eq:4.22}
\begin{gathered}
\mathop{\rm div}\left( \varphi'(\overline{u}) \nabla e_{1} \right)-
d'(\overline{u})\nabla \overline{p} \cdot \nabla p_{1}
-\varphi''(\overline{u}) \nabla \overline{u} \cdot \nabla
e_{1}-g'(\overline{u})\nabla \overline{p} \cdot \nabla e_{1}
= \overline{u}-U,\\
\mathop{\rm div}\left( d(\overline{u}) \nabla p_{1} \right)
+\mathop{\rm div}\left( g(\overline{u}) \nabla e_{1} \right)
= \overline{p}-P, \\
\left. e_{1}\right|_{\partial \Omega}
=\left. p_{1}\right|_{\partial \Omega}  = 0.
\end{gathered}
\end{equation}
It follows again, by \cite{lsu,Lions,sol}, that \eqref{eq:4.22}
has a unique solution $(e_{1}, p_{1})\in W_{2}^{2}(\Omega) \times
W_{2}^{2}(\Omega)$. Since the problem of finding $(e, w)\in W
\times W$ satisfying
\begin{equation}
\label{eq:4.23}
\begin{split}
-\mathop{\rm div}\left( \varphi'(\overline{u}) \nabla e \right)
& - \mathop{\rm div}\left(\varphi''(\overline{u})
e \nabla \overline{u} \right)- \mathop{\rm div}\left(
g(\overline{u}) \nabla w \right)- \mathop{\rm div}\left( g'(\overline{u})e
\nabla \overline{p} \right)\\
&= \mathop{\rm sign}(e_{1}- \overline{e_{1}})
-\mathop{\rm div}\left(d(\overline{u}) \nabla w \right)-
\mathop{\rm div}\left( d'(\overline{u})e \nabla \overline{p} \right)\\
&= \mathop{\rm sign}\left(p_{1}- \overline{p_{1}}\right)
\end{split}
\end{equation}
is uniquely solvable on $W_{2q}^{2}\times W_{2q}^{2}$, choosing
$h=0$ in \eqref{eq:4.21}, multiplying \eqref{eq:4.22} by $(e, w)$,
integrating by parts, and making the difference with
\eqref{eq:4.21}, we obtain
\begin{equation}
\label{eq:4.24}
\begin{gathered}
\int_{\Omega} \left(-\mathop{\rm div}\left( \varphi'(\overline{u}) \nabla e
\right)- \mathop{\rm div}\left( \varphi''(\overline{u}) e \nabla \overline{u}
\right)- \mathop{\rm div}\left( g(\overline{u}) \nabla w \right)
- \mathop{\rm div}\left(g'(\overline{u})e \nabla \overline{p} \right) \right)\\
\times (e_{1}- \overline{e_{1}})\, dx
+\int_{\Omega}\left(-\mathop{\rm div}\left(d(\overline{u}) \nabla w \right)
-\mathop{\rm div}\left( d'(\overline{u})e \nabla \overline{p} \right)
\right)(p_{1}- \overline{p_{1}})\, dx  =0
\end{gathered}
\end{equation}
for all $(e, w) \in W \times W$.
Choosing $(e, w)$ in \eqref{eq:4.24} as the solution of system
\eqref{eq:4.23}, we have
\begin{gather*}
\int_{\Omega} \mathop{\rm sign}(e_{1}- \overline{e_{1}})(e_{1}
- \overline{e_{1}})\, dx + \int_{\Omega}
\mathop{\rm sign}(p_{1}- \overline{p_{1}})(p_{1}-\overline{p_{1}})\, dx = 0.
\end{gather*}
It follows that $e_{1}=\overline{ e_{1}}$ and  $p_{1}=\overline{p_{1}}$.
On the other hand, choosing $(e, w)=(0, 0)$ in \eqref{eq:4.21},
it follows \eqref{eq:4.19}, which concludes the proof of Theorem~\ref{thm5.2}.
\end{proof}

%%%%%%%%%%%%%%%%%%%%%

\section{Conclusion}
\label{con}

In this paper, we considered the optimal control of a steady-state dead oil isotherm problem
with Dirichlet boundary conditions, which is obtained from the standard parabolic-elliptic
system, where the relaxation time for the reduced saturation of oil is very small.
The main purpose was to prove existence and regularity of the optimal control
and then necessary optimality conditions. The proposed method is based on the Lagrangian approach.

%%%%%%%%%%%%%%%%%%%%%

\section*{Acknowledgements}

This work was supported by {\it FEDER} funds through
{\it COMPETE} --- Operational Programme Factors of Competitiveness
(``Programa Operacional Factores de Competitividade'')
and by Portuguese funds through the
{\it Center for Research and Development
in Mathematics and Applications} (University of Aveiro)
and the Portuguese Foundation for Science and Technology
(``FCT --- Funda\c{c}\~{a}o para a Ci\^{e}ncia e a Tecnologia''),
within project PEst-C/MAT/UI4106/2011
with COMPETE number FCOMP-01-0124-FEDER-022690.
Malinowska was supported by Bia{\l}ystok
University of Technology grant S/WI/02/2011.

The authors are very grateful to two anonymous referees
for their valuable comments and helpful suggestions.

%%%%%%%%%%%%%%%%%%%%%

%%%%%%%%%%%%%%%%%%%%%

\end{document}